\documentclass[11pt,a4paper,reqno]{amsart}
\usepackage[a4paper,margin=28mm]{geometry}
\usepackage[T1]{fontenc}
\usepackage{lmodern}
\usepackage{amsmath,amssymb,amsthm}
\usepackage{microtype}
\usepackage{needspace}
\usepackage[colorlinks=true,linkcolor=blue,citecolor=blue,urlcolor=blue]{hyperref}
\newtheorem{maintheorem}{Theorem}

\newtheorem{theorem}{Theorem}[section]
\newtheorem{lemma}[theorem]{Lemma}
\newtheorem{proposition}[theorem]{Proposition}
\numberwithin{equation}{section}
\newcommand{\N}{\mathbb N}

\newcommand{\K}{\mathbb K}
\newcommand{\E}{\mathcal E}
\newcommand{\ind}{\mathbf 1}
\newcommand{\norm}[1]{\lVert #1\rVert}
\newcommand{\abs}[1]{\lvert #1\rvert}
\newcommand{\sot}{\xrightarrow{\mathrm{SOT}}}
\DeclareMathOperator{\Var}{Var}
\allowdisplaybreaks[1]
\hypersetup{pdftitle={On complemented subspaces of L1[0,1]},pdfauthor={Antonio Acuaviva},pdfsubject={Schur and non-Schur complemented subspaces of L1}}


\begin{document}
\title[{On complemented subspaces of $L_1[0,1]$}]{On complemented subspaces of $L_1[0,1]$}
\author[A. Acuaviva]{Antonio Acuaviva}
\address[A.~Acuaviva]{School of Mathematical Sciences, Lancaster University, Lancaster LA1 4YX, United Kingdom}
\email{ahacua@gmail.com}
\date{13 September 2026}
\begin{abstract}
We construct two complemented subspaces of $L_1[0,1]$. The first has the Schur property but fails the Radon--Nikod\'ym property. The second contains a copy of $\ell_2$ but no copy of $L_1[0,1]$. Neither space is isomorphic to a Banach lattice. This gives a negative answer to the complemented-subspace question of Lindenstrauss and Rosenthal. A Lean 4 formalisation of the main results accompanies the paper.
\end{abstract}
\maketitle

\begin{center}
\emph{Companion Lean~4 formalisation:} \url{https://complemented-subspaces-of-l1.github.io/}
\end{center}

\tableofcontents

\section{Introduction}

Understanding the complemented subspaces of classical Banach spaces, and classifying them up to isomorphism, is one of the central problems in the geometry of Banach spaces. For the classical sequence spaces, the corresponding classification has a particularly simple answer. Pe\l czy\'nski~\cite{Pelczynski1960} proved that every infinite-dimensional complemented subspace of $c_0$ is isomorphic to $c_0$, and that every infinite-dimensional complemented subspace of $\ell_p$, $1\le p<\infty$, is isomorphic to $\ell_p$. His decomposition method became a basic tool in the study of complemented subspaces. Lindenstrauss~\cite{Lindenstrauss1967} subsequently established the corresponding result for $\ell_\infty$.

For $1<p<\infty$, $p\ne2$, the familiar complemented subspaces of $L_p[0,1]$ already include the five mutually nonisomorphic spaces~\cite{BRS1981}
\begin{equation*}
\ell_p,\qquad \ell_2,\qquad \ell_p\oplus\ell_2,\qquad \left(\bigoplus_{n=1}^{\infty}\ell_2\right)_{\ell_p},\qquad L_p[0,1].
\end{equation*}
Rosenthal~\cite{Rosenthal1970} showed that these examples do not exhaust the possibilities. For $2<p<\infty$, he constructed new complemented spaces $X_p$ generated by suitable independent mean-zero random variables. Schechtman~\cite{Schechtman1975} then obtained infinitely many isomorphism types. Bourgain, Rosenthal and Schechtman~\cite{BRS1981} developed these constructions further, using an ordinal index to obtain uncountably many pairwise nonisomorphic complemented subspaces of $L_p[0,1]$ for every $1<p<\infty$, $p\ne2$. Their work shows how far the reflexive Lebesgue spaces are from the classification of the sequence spaces, although a full classification of their complemented subspaces remains unavailable.

The spaces $L_1[0,1]$ and $C[0,1]$ remain the most prominent examples of classical Banach spaces whose complemented subspaces are not yet fully understood. For $L_1[0,1]$, the question of Lindenstrauss and Rosenthal~\cite{LR} asks whether every infinite-dimensional complemented subspace is isomorphic to either $\ell_1$ or $L_1[0,1]$. This is the separable case of Pe\l czy\'nski's question~\cite{Pelczynski1960} about complemented subspaces of $L_1$-spaces. Popov~\cite{Popov} surveys the problem and its partial solutions.

The metric version has a complete answer. Over the real field, Grothendieck's theorem~\cite[Proposition~1 and Theorem~1]{Grothendieck1955} shows that the range of a contractive projection on an $L_1$-space is linearly isometric to an $L_1$-space. Douglas~\cite[Theorem~3]{Douglas1965} proved the converse for subspaces of $L_1[0,1]$: every subspace linearly isometric to an $L_1$-space is the range of a contractive projection. The characterization holds over both scalar fields; see also~\cite[Theorem~4.1]{BernauLacey1974}. If the projection is also positive, its range is a closed sublattice~\cite[Theorem~2]{Douglas1965}.

Positivity alone, without a norm-one assumption, is enough for the isomorphic classification. The range of a positive projection on a Banach lattice admits an equivalent Banach lattice norm~\cite[Proposition~III.11.5]{Schaefer1974}. A complemented subspace of an $L_1$-space that is isomorphic to a Banach lattice is necessarily isomorphic to an $L_1$-space \cite[Corollary~2.3]{DeHeviaLattices}. Consequently, every infinite-dimensional range of a positive projection on $L_1[0,1]$ is isomorphic to $\ell_1$ or $L_1[0,1]$.

Pe\l czy\'nski's work~\cite{Pelczynski1960} implies that every infinite-dimensional complemented subspace of $L_1[0,1]$ contains a copy of $\ell_1$ complemented in $L_1[0,1]$. Lindenstrauss and Pe\l czy\'nski~\cite{LP1968} proved that such a subspace is isomorphic to $\ell_1$ if it has an unconditional basis, and Lewis and Stegall~\cite{LS} reached the same conclusion under the Radon--Nikod\'ym property. In the other direction, Enflo and Starbird~\cite{ES1979} proved that a complemented subspace containing a copy of $L_1[0,1]$ must itself be isomorphic to $L_1[0,1]$. Thus a counterexample must lie outside both of these established parts of the classification.

Popov~\cite{Popov} separates the problem into two questions: whether every infinite-dimensional complemented Schur subspace of $L_1[0,1]$ is isomorphic to $\ell_1$, and whether every complemented subspace without the Schur property is isomorphic to $L_1[0,1]$. The constructions below give negative answers to both questions. The first has the Schur property and fails the Radon--Nikod\'ym property, so it also shows why the latter hypothesis in the Lewis--Stegall theorem cannot be replaced by the Schur property.

\begin{maintheorem}\label{thm:main}
There is a Dunford--Pettis projection $P$ on $L_1[0,1]$ with $\norm{P}\le4+3\sqrt2$. Its range has the Schur property but fails the Radon--Nikod\'ym property. In particular, it is not isomorphic to any Banach lattice.
\end{maintheorem}

Changing the parameters in the same construction gives the non-Schur example. Its proof is given in Section~\ref{sec:non-schur}.

\begin{maintheorem}\label{thm:non-schur}
There is a projection $Q$ on $L_1[0,1]$, with $\norm{Q}\le4+3\sqrt2$, whose range contains an isomorphic copy of $\ell_2$ but no isomorphic copy of $L_1[0,1]$. In particular, its range does not have the Schur property and is not isomorphic to any Banach lattice.
\end{maintheorem}

Both constructions are based on product measures on the countable sign group, a setting with antecedents in Rosenthal's biased-coin construction~\cite{Rosenthal1976}; see also~\cite[Section~2.4]{Popov}. In the Schur example, we use singular convolution measures whose operators are Dunford--Pettis, a combination already present in work of Cost\'e and Uhl. Schachermayer~\cite[Introduction and Section~3]{Schachermayer1986} discusses these antecedents and constructs singular measures whose convolution operators are compact on $L_2$. This compactness is also the ingredient we use to establish the Dunford--Pettis property of the blocks of our projection.

With suitable choices of the sequence $\lambda=(\lambda_j)_{j\in\N}$, the general construction in Section~\ref{sec:general-construction} also yields continuum many pairwise nonisomorphic complemented subspaces of $L_1[0,1]$ and $2^{\mathfrak c}$ norm-closed two-sided ideals in $\mathcal L(L_1[0,1])$, where $\mathfrak c=2^{\aleph_0}$. Details of this construction will appear in a subsequent version.

Section~\ref{sec:general-construction} begins with probability distributions of infinitely many independent biased signs. These form a convolution semigroup; see Lemma~\ref{lem:biased-signs}. We take their exponential averages, called resolvent measures; Proposition~\ref{prop:resolvents} gives the identities satisfied by the associated convolution operators. We then define a function space by summing weighted errors between a function and its averaged versions. Lemma~\ref{lem:intrinsic} establishes its basic properties, and Proposition~\ref{prop:projection} realises it as a complemented subspace of an $L_1$-space.

For the Schur example, two properties of these averages play distinct roles. By Lemma~\ref{lem:carrier}, all the distributions give full measure to the same Haar-null set, so their exponential averages remain singular. Proposition~\ref{prop:compact-resolvents} shows that the associated convolution operators are nevertheless compact on $L_2$. This compactness yields the Schur property through Liu's band theorem~\cite[Lemma~2.17]{Liu1998}, as shown in Proposition~\ref{prop:schur}. Singularity obstructs the Radon--Nikod\'ym property through Cost\'e's nonrepresentability result in Lemma~\ref{lem:singular}, as explained in the proof of Theorem~\ref{thm:main}. The lattice obstruction follows from the structural result of de Hevia et al.~\cite[Corollary~2.3]{DeHeviaLattices}, which is also used for Theorem~\ref{thm:non-schur}.

In Section~\ref{sec:non-schur}, we use the same bias in every coordinate. The general projection applies, and the coordinate functions span a copy of $\ell_2$. Bonami's estimate, the Enflo--Starbird criterion and Liu's second band theorem~\cite[Lemma~2.14]{Liu1998} exclude copies of $L_1$ from the range, proving Theorem~\ref{thm:non-schur}. Section~\ref{subsec:interpolation} gives an interpolation interpretation of both constructions.

\section{Preliminaries}

We write $\N=\{1,2,\ldots\}$ and $\N_0=\{0,1,2,\ldots\}$. All Banach spaces are over a fixed field $\K\in\{\mathbb R,\mathbb C\}$, and an operator means a bounded linear map. For a Banach space $Z$, $\mathcal L(Z)$ denotes the Banach space of operators $T\colon Z\to Z$, equipped with the operator norm. The symbols $B_Z$ and $S_Z$ denote the closed unit ball and unit sphere of $Z$, respectively. The identity operator on $Z$ is denoted by $I_Z$, or simply $I$ when the space is clear. The measures underlying $L_p$-spaces are completed; products and convolutions are formed from Borel measures.

We use standard Banach lattice terminology; see Schaefer~\cite{Schaefer1974} and Meyer-Nieberg~\cite{MeyerNieberg1991}. We recall a few notions used below. A real Banach lattice $Z$ is \emph{order complete} if every nonempty subset that is bounded above has a supremum. An \emph{order ideal} in $Z$ is a vector subspace $I\subset Z$ such that $x\in I$ and $\abs{y}\le\abs{x}$ imply $y\in I$. A \emph{band} is an order ideal closed under suprema of increasing nets of its positive elements, whenever these suprema exist in $Z$. An operator $T$ between real Banach lattices is \emph{positive} if $Tx\ge0$ whenever $x\ge0$, and $S\le T$ means that $T-S$ is positive. For a real $L_1$-space $Z$, this order makes $\mathcal L(Z)$ an order-complete Banach lattice; see~\cite[p.~5]{Liu1998}.

An operator is \emph{Dunford--Pettis} if it sends weakly null sequences to norm-null sequences. A Banach space has the \emph{Schur property} if every weakly null sequence in it is norm null.

\Needspace{7\baselineskip} Let $(S,\eta)$ be a probability space and let $Z$ be a Banach space. An operator $T\colon L_1(S,\eta)\to Z$ is \emph{representable} if there is an essentially bounded, strongly measurable function $b\colon S\to Z$ with
\begin{equation*}
Tf=\int_S f(s)b(s)\,d\eta(s)\qquad(f\in L_1(S,\eta)),
\end{equation*}
where the integral is a Bochner integral. Postcomposition with an operator preserves representability.

We will obtain the Dunford--Pettis property by checking compactness on bounded functions. We use the standard compact-restriction criterion in the following form; see~\cite[Fact~1.1]{Girardi}.

\begin{lemma}\label{lem:dp}
Let $(S,\eta)$ be a standard nonatomic probability space. If $T\colon L_1(S,\eta)\to Z$ has compact restriction to $L_\infty(S,\eta)$, then $T$ is Dunford--Pettis.
\end{lemma}

\Needspace{8\baselineskip} We work on the compact group $G=\{-1,1\}^{\N}$, with coordinatewise multiplication and Haar probability measure $m$. For a finite scalar Borel measure $\tau$ on $G$, the \emph{convolution operator} $T_\tau\colon L_1(G,m)\to L_1(G,m)$ is defined by
\begin{equation*}
T_\tau f(x)=(f*\tau)(x)=\int_G f(xy^{-1})\,d\tau(y), \qquad f\in L_1(G,m).
\end{equation*}
When $\tau$ is a probability measure, $T_\tau f$ averages translates of $f$ according to $\tau$. Haar invariance and Fubini's theorem make this definition independent of the Borel representative of $f$, and $\norm{T_\tau}\le\norm{\tau}$, where the latter norm is total variation. For finite $F\subset\N$, the \emph{Walsh character} $\chi_F(x)=\prod_{j\in F}x_j$ satisfies $T_\tau\chi_F=\widehat\tau(F)\chi_F$, where $\widehat\tau(F)=\int_G\chi_F\,d\tau$. Finite linear combinations of Walsh characters are called \emph{Walsh polynomials}. They are dense in $L_p(G,m)$ for $1\le p<\infty$, and the Walsh characters form an orthonormal basis of $L_2(G,m)$.

Cost\'e's theorem says that a scalar convolution operator on $L_1$ of a compact metrizable abelian group is representable if and only if its measure is absolutely continuous with respect to Haar measure; see~\cite[Theorem~1.1(c)--(d)]{RobderaSaab1998}. We use the following consequence.

\begin{lemma}\label{lem:singular}
If $\tau$ is a nonzero finite scalar Borel measure on $G$, singular with respect to $m$, then $T_\tau\colon L_1(G,m)\to L_1(G,m)$ is not representable.
\end{lemma}

A family $(\mu_u)_{u\ge0}$ of probability measures on $G$ is a \emph{convolution semigroup} if $\mu_0=\delta_{\mathbf1}$ and $\mu_u*\mu_v=\mu_{u+v}$ for $u,v\ge0$. We assume that $u\mapsto\mu_u(D)$ is measurable for every Borel set $D$. The associated operator semigroup $(C_u)_{u\ge0}$, where $C_u=T_{\mu_u}$, is assumed to be \emph{strongly continuous} on $L_1(G,m)$, meaning that $C_uf\to f$ in $L_1(G,m)$ as $u\downarrow0$ for every $f\in L_1(G,m)$. For $t>0$, the \emph{normalised resolvent measure} and its convolution operator are
\begin{equation}\label{eq:nu}
\nu_t(D)=\int_0^\infty\mu_u(D)t^{-1}e^{-u/t}\,du, \qquad R_t=T_{\nu_t}.
\end{equation}
Thus, for $f\in L_1(G,m)$,
\begin{equation*}
R_tf=\int_0^\infty t^{-1}e^{-u/t}C_uf\,du,
\end{equation*}
where the latter integral is a Bochner integral. The exponential weight has integral one, so $\nu_t$ is a probability measure and $R_t$ is a positive contraction on every $L_p(G,m)$. If $-A$ denotes the generator of $(C_u)_{u\ge0}$ on $L_1(G,m)$, then
\begin{equation*}
R_t=t^{-1}(t^{-1}I+A)^{-1}=(I+tA)^{-1}.
\end{equation*}
The factor $t^{-1}$ accounts for the normalisation of the usual resolvent at $t^{-1}$. We refer to $R_t$ simply as a resolvent below.

\section{The general construction}\label{sec:general-construction}

Fix a sequence $\lambda=(\lambda_j)_{j\in\N}$ of nonnegative real numbers. We use this sequence to construct probability distributions of infinitely many independent biased signs. We shall specify $\lambda$ in Sections~\ref{sec:schur} and~\ref{sec:non-schur} to obtain the Schur and non-Schur examples, respectively. For finite $F\subset\N$, put $\lambda_F=\sum_{j\in F}\lambda_j$. For $u\ge0$, set
\begin{equation}\label{eq:product-signs}
p_j(u)=\frac{1+e^{-u\lambda_j}}2\quad(j\in\N),\qquad \mu_u=\bigotimes_{j\in\N} \bigl(p_j(u)\delta_1+(1-p_j(u))\delta_{-1}\bigr).
\end{equation}
Each factor is a probability measure on $\{-1,1\}$, so the countable product-measure theorem gives a unique Borel probability measure $\mu_u$ on $G$. Under $\mu_u$, the coordinates are independent, and the $j$th sign equals $1$ with probability $p_j(u)$ and $-1$ with probability $1-p_j(u)$. In particular, its mean is
\begin{equation}\label{eq:bias}
\int_G x_j\,d\mu_u(x)=2p_j(u)-1=e^{-u\lambda_j} \qquad(j\in\N,\ u\ge0).
\end{equation}

A \emph{cylinder set} in $G$ is a set whose membership depends on only finitely many coordinates. For finite $F\subset\N$ and $\varepsilon=(\varepsilon_j)_{j\in F}\in\{-1,1\}^F$, the \emph{basic cylinder}
\begin{equation*}
D(F,\varepsilon)=\{x\in G:x_j=\varepsilon_j\text{ for every }j\in F\}
\end{equation*}
prescribes the signs on $F$ and leaves all other coordinates unrestricted. Every cylinder set is a finite disjoint union of basic cylinders with a common coordinate set. For each $u\ge0$ and cylinder set $D$, the number $\mu_u(D)$ is called the \emph{cylinder probability} of $D$ under $\mu_u$. For a basic cylinder, independence gives
\begin{equation}\label{eq:cylinder-probability}
\mu_u\bigl(D(F,\varepsilon)\bigr) =\prod_{j\in F}\frac{1+\varepsilon_j e^{-u\lambda_j}}2.
\end{equation}
The cylinder sets form an algebra generating the Borel $\sigma$-algebra of $G$.

A family $(\mu_u)_{u\ge0}$ of Borel probability measures on $G$ is a \emph{probability kernel} if $u\mapsto\mu_u(D)$ is Borel measurable on $[0,\infty)$ for every Borel set $D\subset G$.

\begin{lemma}\label{lem:biased-signs}
The measures in~\eqref{eq:product-signs} and the operators $C_u=T_{\mu_u}$ have the following properties.
\begin{enumerate}
\renewcommand{\theenumi}{\roman{enumi}}
\renewcommand{\labelenumi}{(\theenumi)}
\item\label{item:biased-kernel}
For every cylinder set $D$, the map $u\mapsto\mu_u(D)$ is continuous on $[0,\infty)$, and $(\mu_u)_{u\ge0}$ is a probability kernel.
\item\label{item:biased-semigroup}
We have $\mu_0=\delta_{\mathbf1}$ and
\begin{equation*}
\mu_u*\mu_v=\mu_{u+v}\qquad(u,v\ge0).
\end{equation*}
\item\label{item:biased-contractions}
Each $C_u$ is a positive contraction on $L_p(G,m)$ for $1\le p\le\infty$.
\item\label{item:biased-multiplier}
For finite $F\subset\N$,
\begin{equation*}
C_u\chi_F=e^{-u\lambda_F}\chi_F.
\end{equation*}
\item\label{item:biased-continuity}
The semigroup $(C_u)_{u\ge0}$ is strongly continuous on $L_p(G,m)$ for every $1\le p<\infty$.
\end{enumerate}
\end{lemma}
\begin{proof}
For~(\ref{item:biased-kernel}), the right-hand side of~\eqref{eq:cylinder-probability} is a finite product of continuous functions of $u$. Finite disjoint unions therefore give continuity of $u\mapsto\mu_u(D)$ for every cylinder set $D$. Let $\mathcal D$ be the collection of Borel sets $D\subset G$ for which $u\mapsto\mu_u(D)$ is Borel measurable. This collection contains the cylinder sets and is closed under complements, since $\mu_u(G\setminus D)=1-\mu_u(D)$. If $(D_n)_{n\in\N}$ is a pairwise disjoint sequence in $\mathcal D$, then
\begin{equation*}
\mu_u\left(\bigcup_{n\in\N}D_n\right) =\sum_{n=1}^{\infty}\mu_u(D_n)
\end{equation*}
is measurable as a pointwise limit of measurable partial sums. Thus $\mathcal D$ is a Dynkin system containing the cylinder algebra. The $\pi$--$\lambda$ theorem gives $\mathcal D=\mathcal B(G)$, proving the probability-kernel property.

For~(\ref{item:biased-semigroup}), at $u=0$ every coordinate is $1$ almost surely, so $\mu_0=\delta_{\mathbf1}$. Take independent random elements $X=(X_j)_{j\in\N}$ and $Y=(Y_j)_{j\in\N}$ of $G$, with probability distributions $\mu_u$ and $\mu_v$, respectively. Their coordinatewise product $XY$ has probability distribution $\mu_u*\mu_v$. Its coordinates are independent, and
\begin{equation*}
\mathbb E[X_jY_j]=\mathbb E[X_j]\mathbb E[Y_j] =e^{-u\lambda_j}e^{-v\lambda_j} =e^{-(u+v)\lambda_j}\qquad(j\in\N).
\end{equation*}
The probability distribution of a sign is determined by its mean, so $XY$ has probability distribution $\mu_{u+v}$.

For~(\ref{item:biased-contractions}), positivity of $C_u$ follows from its definition as an average of translates. For $1\le p<\infty$, Jensen's inequality and Haar invariance give
\begin{equation*}
\norm{C_uf}_p^p \le\int_G\int_G|f(xy^{-1})|^p\,d\mu_u(y)\,dm(x) =\norm{f}_p^p.
\end{equation*}
For $p=\infty$, the same averaging formula gives $\norm{C_uf}_\infty\le\norm{f}_\infty$.

For~(\ref{item:biased-multiplier}), independence and~\eqref{eq:bias} give
\begin{equation*}
C_u\chi_F(x) =\chi_F(x)\int_G\prod_{j\in F}y_j\,d\mu_u(y) =\chi_F(x)\prod_{j\in F}e^{-u\lambda_j}.
\end{equation*}

For~(\ref{item:biased-continuity}), part~(\ref{item:biased-multiplier}) gives $C_uw\to w$ in $L_p(G,m)$ as $u\downarrow0$ for each Walsh polynomial $w$ and $1\le p<\infty$. For arbitrary $f\in L_p(G,m)$, the contractivity in~(\ref{item:biased-contractions}) gives
\begin{equation*}
\norm{C_uf-f}_p \le 2\norm{f-w}_p+\norm{C_uw-w}_p.
\end{equation*}
Choose $w$ so that the first term is arbitrarily small, and then let $u\downarrow0$. Density of Walsh polynomials proves convergence for every $f$. Parts~(\ref{item:biased-semigroup}) and~(\ref{item:biased-contractions}) extend continuity from zero to every $u\ge0$.
\end{proof}

Let $\nu_t$ and $R_t$ be the resolvent measures and operators defined in~\eqref{eq:nu}, and write $S_n=R_{2^{-n}}$ for $n\in\N_0$. The following properties will be used in the construction of the projection.

\begin{proposition}\label{prop:resolvents}
For every finite $F\subset\N$ and $t>0$,
\begin{equation}\label{eq:resolvent-multiplier}
R_t\chi_F=(1+t\lambda_F)^{-1}\chi_F.
\end{equation}
The operators $R_t$ commute, and $R_t\sot I$ on $L_1(G,m)$ as $t\downarrow0$. For $s,t>0$,
\begin{equation}\label{eq:resolvent}
(I-R_s)R_t=\frac{s}{t}R_s(I-R_t),\qquad \norm{(I-R_s)R_t}\le2\min\{1,s/t\}.
\end{equation}
Moreover,
\begin{equation*}
S_{n-1}(I-S_n)=S_n-S_{n-1}\qquad(n\ge1).
\end{equation*}
\end{proposition}
\begin{proof}
Formula~\eqref{eq:resolvent-multiplier} follows from part~(\ref{item:biased-multiplier}) of Lemma~\ref{lem:biased-signs} and~\eqref{eq:nu}. It also gives commutation on Walsh polynomials, and hence on $L_1(G,m)$ by density and boundedness. For each finite $F\subset\N$, the multiplier $(1+t\lambda_F)^{-1}$ tends to $1$ as $t\downarrow0$. Hence $R_tw\to w$ in $L_1(G,m)$ as $t\downarrow0$ for every Walsh polynomial $w$. For arbitrary $f\in L_1(G,m)$ and any Walsh polynomial $w$, contractivity gives
\begin{equation*}
\norm{R_tf-f}_1 \le\norm{R_t(f-w)}_1+\norm{R_tw-w}_1+\norm{w-f}_1 \le2\norm{f-w}_1+\norm{R_tw-w}_1.
\end{equation*}
By density, we can first choose $w$ so that $\norm{f-w}_1$ is arbitrarily small and then let $t\downarrow0$. Thus $R_tf\to f$ in $L_1(G,m)$ as $t\downarrow0$. The two identities hold on each Walsh character by the corresponding scalar identities, and therefore on every Walsh polynomial by linearity. By density and boundedness of the operators, they hold on all of $L_1(G,m)$. The bound follows from contractivity and $\norm{I-R_t}\le2$.
\end{proof}

We first define the space $\mathcal{X}_\lambda$ and then realise it isometrically as a complemented subspace of an $L_1$-space. Set
\begin{equation}\label{eq:intrinsic}
\mathcal{X}_\lambda=\{f\in L_1(G,m):\norm{f}_{\mathcal{X}_\lambda}<\infty\},\qquad \norm{f}_{\mathcal{X}_\lambda}=\norm{f}_1+ \sum_{n=1}^\infty2^{n/2}\norm{(I-S_n)f}_1.
\end{equation}
Let $E_m$ be conditional expectation onto the first $m$ coordinates of $G$, with $E_0f=(\int_G f\,dm)\mathbf1$. We begin by proving some elementary properties of $\mathcal{X}_\lambda$.

\begin{lemma}\label{lem:intrinsic}
The space $\mathcal{X}_\lambda$ is a separable Banach space. Every finite-coordinate function belongs to $\mathcal{X}_\lambda$, and these functions are dense in $\mathcal{X}_\lambda$. The restrictions of $E_m$ to $\mathcal{X}_\lambda$ are contractive finite-rank projections and $E_mf\to f$ in $\mathcal{X}_\lambda$ for every $f\in \mathcal{X}_\lambda$.
\end{lemma}
\begin{proof}
A Cauchy sequence in $\mathcal{X}_\lambda$ has a limit $f$ in $L_1(G,m)$ and limits in the associated $\ell_1$-sum of difference coordinates. Boundedness of each $I-S_n$ identifies these limits with $(I-S_n)f$, proving completeness. On any fixed finite-dimensional span of Walsh characters, $\norm{I-S_n}=O(2^{-n})$, so every finite-coordinate function belongs to $\mathcal{X}_\lambda$. The operators $E_m$ commute with $S_n$ and are contractive on $L_1(G,m)$, giving $\norm{E_mf}_{\mathcal{X}_\lambda}\le\norm{f}_{\mathcal{X}_\lambda}$. For fixed $f\in \mathcal{X}_\lambda$, every coordinate of its norm in~\eqref{eq:intrinsic} converges under $E_m$, while the difference-coordinate tails are bounded by twice the corresponding summable tail for $f$. Thus $E_mf\to f$ in $\mathcal{X}_\lambda$. This also proves density and separability.
\end{proof}

Put $\E=\ell_1\bigl(\N_0;L_1(G,m)\bigr)$. This is an $L_1$-space: give $\Omega=\N_0\times G$ the probability measure $\rho(\{j\}\times D)=p_jm(D)$, where $p_j=2^{-j-1}$, and use the surjective lattice isometry
\begin{equation*}
U\colon L_1(\Omega,\rho)\to\E,\qquad (Uh)_j=p_jh(j,\cdot) \quad(j\in\N_0).
\end{equation*}
Partitioning $[0,1]$ into intervals of lengths $(p_j)_{j\in\N_0}$ and identifying each normalised interval with $G$ by binary digits also gives a lattice isometry $\E\cong L_1[0,1]$.

We will use this identification to apply two band theorems of Liu to the blocks of our projections. Write $\E_{\mathbb R}$ for the real $L_1$-space just described and $\Pi_i$ for its $i$th coordinate projection, $i\in\N_0$. The next lemma shows that an operator belongs to a band whenever all of its blocks do.

\begin{lemma}\label{lem:block-band}
Let $\mathcal B$ be a band in $\mathcal L(\E_{\mathbb R})$. If $T\in\mathcal L(\E_{\mathbb R})$ and $\Pi_iT\Pi_j\in\mathcal B$ for all $i,j\in\N_0$, then $T\in\mathcal B$.
\end{lemma}
\begin{proof}
The map $S\mapsto\Pi_iS\Pi_j$ is a positive projection dominated by the identity on the operator lattice, hence a band projection. It therefore preserves moduli: $|\Pi_iT\Pi_j|=\Pi_i|T|\Pi_j$. For $N\in\N_0$, put $H_N=\sum_{i=0}^N\Pi_i$. Then
\begin{equation*}
\sum_{i,j=0}^N|\Pi_iT\Pi_j|=H_N|T|H_N.
\end{equation*}
Since each summand is positive and $0\le H_N\le I$, these sums increase in the operator order and are bounded above by $|T|$. Moreover, $H_N|T|H_Nf\to|T|f$ in norm as $N\to\infty$ for every $f\ge0$, so their supremum is $|T|$. Each finite sum belongs to $\mathcal B$, so closure under order suprema gives $|T|$, and hence $T$, in $\mathcal B$.
\end{proof}

The definition of the norm on $\mathcal{X}_\lambda$ gives the isometric embedding
\begin{equation}\label{eq:embedding}
J\colon \mathcal{X}_\lambda\to\E,\qquad (Jf)_0=f,\qquad (Jf)_n=2^{n/2}(I-S_n)f\quad(n\ge1).
\end{equation}
To construct a projection onto its image, define $K\colon\E\to L_1(G,m)$ by
\begin{equation}\label{eq:retraction}
Kx=S_0x_0+\sum_{n=1}^\infty2^{-n/2}S_{n-1}x_n \qquad(x=(x_n)_{n\in\N_0}\in\E).
\end{equation}
Since the operators $S_n$, $n\in\N_0$, are contractions, the series converges absolutely in $L_1(G,m)$ and $\norm{K}_{\E\to L_1(G,m)}\le1$.

\begin{proposition}\label{prop:projection}
The operator $K$ takes values in $\mathcal{X}_\lambda$. As a map $K\colon\E\to\mathcal{X}_\lambda$, where the latter is equipped with the norm $\norm{\cdot}_{\mathcal{X}_\lambda}$ defined in~\eqref{eq:intrinsic}, it satisfies $\norm{K}_{\E\to\mathcal{X}_\lambda}\le4+3\sqrt2$ and $KJ=I_{\mathcal{X}_\lambda}$. Consequently $P=JK$ is a projection on $\E$ with $\norm{P}\le4+3\sqrt2$ and range $J\mathcal{X}_\lambda$. Its zeroth diagonal block is $P_{00}=R_1$.
\end{proposition}
\begin{proof}
Set $A_0=I$, $B_0=S_0$, and $A_i=I-S_i$, $B_i=S_{i-1}$ for $i\ge1$. Define the operators
\begin{equation}\label{eq:blocks}
P_{ij}=2^{(i-j)/2}A_iB_j\qquad(i,j\in\N_0).
\end{equation}
We shall show that these are the blocks of the bounded projection $P=JK$. Put $r=2^{-1/2}$. For $j\ge1$, the zeroth block has norm at most $r^j$, and~\eqref{eq:resolvent} gives
\begin{equation*}
\norm{P_{ij}}\le\begin{cases}
    2r^{j-i},&1\le i<j,\\
    r^{i-j},&i\ge j.
\end{cases}
\end{equation*}
It follows that
\begin{equation}\label{eq:column}
\sum_{i=0}^\infty\norm{P_{ij}} \le r^j+2\sum_{k=1}^{j-1}r^k+\sum_{k=0}^\infty r^k \le\frac{1+2r}{1-r}=4+3\sqrt2.
\end{equation}
For $j=0$, the bounds are $\norm{P_{00}}\le1$ and $\norm{P_{i0}}\le2r^i$ for $i\ge1$, so that the column sum is at most $3+2\sqrt2$. For finitely supported $x\in\E$, these estimates give $\norm{Kx}_{\mathcal{X}_\lambda}\le(4+3\sqrt2)\norm{x}_{\E}$. Since the finitely supported sequences are dense in $\E$ and $\mathcal{X}_\lambda$ is complete, this restriction extends uniquely to a bounded operator from $\E$ to $\mathcal{X}_\lambda$ with norm at most $4+3\sqrt2$. Composing with the continuous inclusion $\mathcal{X}_\lambda\hookrightarrow L_1(G,m)$ recovers the original operator $K$ by density. From now on, we regard $K$ as a bounded operator $K\colon\E\to\mathcal{X}_\lambda$. For $f\in \mathcal{X}_\lambda$, the series defining $KJf$ is absolutely convergent and its partial sums are
\begin{equation*}
S_0f+\sum_{n=1}^N S_{n-1}(I-S_n)f =S_Nf\longrightarrow f\quad\text{in }L_1(G,m).
\end{equation*}
Thus $KJ=I_{\mathcal{X}_\lambda}$, and $P^2=JKJK=JK=P$ with range $J\mathcal{X}_\lambda$. Finally, $P_{00}=A_0B_0=S_0=R_1$.
\end{proof}

\section{The Schur example}\label{sec:schur}

We now choose
\begin{equation*}
\lambda_j=\sqrt{\log(j+1)}\qquad(j\in\N),
\end{equation*}
and denote the corresponding space $\mathcal{X}_\lambda$ by $\mathcal{CS}$. Throughout this section, the measures and operators from Section~\ref{sec:general-construction} refer to this choice of $\lambda$. We will show that, for this choice, the resolvent measures are singular with respect to Haar measure, while their convolution operators are compact on $L_2$; see Schachermayer~\cite[Section~3]{Schachermayer1986} for a precedent.

A \emph{common carrier} for $(\mu_u)_{u>0}$ is a Borel set $B\subset G$ such that $\mu_u(B)=1$ for every $u>0$. The next lemma gives a Haar-null common carrier, so the exponential averages are singular as well. 

\begin{lemma}\label{lem:carrier}
There is a Borel set $N\subset G$ such that $m(N)=0$ and $\mu_u(N)=1$ for every $u>0$. 
\end{lemma}
\begin{proof}
For $r\ge1$, put
\begin{equation*}
Z_r(x)=2^{-r}\sum_{j=1}^{2^r}x_j,\qquad a_r=2^{-r/4},\qquad N=\bigcup_{r_0\in\N}\bigcap_{r\ge r_0}\{Z_r>a_r\}.
\end{equation*}
Thus $x\in N$ if there is a starting index $r_0\in\N$, possibly depending on $x$, such that $Z_r(x)>a_r$ for every $r\ge r_0$. Under $m$, the mean of $Z_r$ is zero and its variance is $2^{-r}$. By Chebyshev's inequality, for every $r\in\N$,
\begin{equation*}
m\{Z_r>a_r\} \le m\{|Z_r|>a_r\} =m\{|Z_r-\mathbb E_m[Z_r]|>a_r\} \le\frac{\Var_m(Z_r)}{a_r^2} =\frac{2^{-r}}{2^{-r/2}} =2^{-r/2}.
\end{equation*}
Consequently,
\begin{equation*}
\sum_{r=1}^{\infty}m\{Z_r>a_r\} \le\sum_{r=1}^{\infty}2^{-r/2}<\infty.
\end{equation*}
Borel--Cantelli gives $m(N)=0$. For fixed $u>0$, the mean $b_r(u)=\int_G Z_r\,d\mu_u$ satisfies
\begin{equation*}
b_r(u)\ge e^{-u\sqrt{\log(2^r+1)}} \ge e^{-u\sqrt{(r+1)\log2}},\qquad \Var_{\mu_u}(Z_r)\le2^{-r}.
\end{equation*}
Thus $b_r(u)/a_r\to\infty$ as $r\to\infty$. Applying Chebyshev's inequality again gives
\begin{equation*}
\sum_{r=1}^{\infty}\mu_u\{|Z_r-b_r(u)|>b_r(u)/2\} \le4\sum_{r=1}^{\infty}2^{-r}e^{2u\sqrt{(r+1)\log2}}<\infty.
\end{equation*}
Another application of Borel--Cantelli gives $\mu_u(N)=1$.
\end{proof}

By~\eqref{eq:nu} and Lemma~\ref{lem:carrier}, $\nu_t(N)=1$ for every $t>0$, so each $\nu_t$ is singular with respect to $m$.

\begin{proposition}\label{prop:compact-resolvents}
For every $t>0$, the operator $R_t$ is compact on $L_2(G,m)$. Its restriction $R_t\colon L_\infty(G,m)\to L_1(G,m)$ is also compact.
\end{proposition}
\begin{proof}
The Walsh characters form an orthonormal basis of $L_2(G,m)$ and, by~\eqref{eq:resolvent-multiplier}, are eigenvectors of $R_t$ with eigenvalues $(1+t\lambda_F)^{-1}$. For every $0<L<\infty$, only finitely many $F$ satisfy $\lambda_F\le L$, since such $F$ are contained in $\{j\in\N:\lambda_j\le L\}$, which is finite because $\lambda_j\to\infty$ as $j\to\infty$. Setting all other eigenvalues to zero gives a finite-rank operator whose distance from $R_t$ in the $L_2$-operator norm is at most $(1+tL)^{-1}$. Letting $L\to\infty$ proves compactness on $L_2(G,m)$. Composing with $L_\infty\hookrightarrow L_2\hookrightarrow L_1$ gives compactness of the restriction to $L_\infty$.
\end{proof}

Every block of the projection has compact restriction to bounded functions. Liu's theorem that Dunford--Pettis operators form a band \cite[Lemma~2.17]{Liu1998} allows us to assemble these blocks using Lemma~\ref{lem:block-band}.

\begin{proposition}\label{prop:schur}
The projection $P$ is Dunford--Pettis, and $\mathcal{CS}$ has the Schur property.
\end{proposition}
\begin{proof}
First work over the reals. By Proposition~\ref{prop:compact-resolvents}, each $B_j$ is compact on $L_2$, and each $A_i$ is bounded there. Thus every block $P_{ij}\colon L_1(G,m)\to L_1(G,m)$ has compact restriction to $L_\infty(G,m)$ and is Dunford--Pettis by Lemma~\ref{lem:dp}. The extended blocks $\Pi_iP\Pi_j$ are Dunford--Pettis as well, since bounded pre- and postcomposition preserve this property. Liu's band theorem and Lemma~\ref{lem:block-band} therefore make $P$ Dunford--Pettis. Over the complex field, $P$ is the complexification of this real operator. Applying the real result to the real and imaginary parts of a weakly null sequence gives the same conclusion. The restriction of $P$ to $J\mathcal{CS}$ is the identity, so $J\mathcal{CS}$, and therefore $\mathcal{CS}$, has the Schur property.
\end{proof}

We can now complete the first construction by proving that its range fails the Radon--Nikod\'ym property. The singular measure defining $P_{00}=R_1$ provides the obstruction: by Lemma~\ref{lem:singular}, this convolution operator is not representable.

\begin{proof}[Proof of Theorem~\ref{thm:main}]
Propositions~\ref{prop:projection} and~\ref{prop:schur} give the projection, its bound and the Schur property of its range $J\mathcal{CS}$, after identifying $\E$ with $L_1[0,1]$. The map $R_1\colon L_1(G,m)\to\mathcal{CS}$ is bounded, since it is the restriction of $K$ to the zeroth coordinate. If $\mathcal{CS}$ had the Radon--Nikod\'ym property, this map would be representable by the standard operator characterization of that property; see~\cite[p.~61]{Girardi}. Postcomposing with the continuous inclusion $\mathcal{CS}\hookrightarrow L_1(G,m)$ would represent convolution by $\nu_1$, contrary to Lemmas~\ref{lem:carrier} and~\ref{lem:singular}. Thus $\mathcal{CS}$, and therefore $J\mathcal{CS}$, fails the Radon--Nikod\'ym property.

Over the reals, if $J\mathcal{CS}$ were isomorphic to a Banach lattice, \cite[Corollary~2.3]{DeHeviaLattices} would make it isomorphic to an $L_1$-space. Such a space, being separable and infinite dimensional, is either isomorphic to $\ell_1$ or contains a copy of $L_1[0,1]$. The Schur property excludes the latter alternative, whereas failure of the Radon--Nikod\'ym property excludes $\ell_1$.

For complex scalars, the underlying real space of $\E$ is isomorphic to $\E_{\mathbb R}\oplus_1\E_{\mathbb R}$. Thus the underlying real range is complemented in a real $L_1$-space and still has the Schur property and fails the Radon--Nikod\'ym property. Since the underlying real space of a complex Banach lattice is real-isomorphic to the product of two copies of its real part, the same argument excludes an isomorphism with a complex Banach lattice.
\end{proof}

\section{The non-Schur example}\label{sec:non-schur}

We now set $\lambda_j=1$ for every $j\in\N$. Lemma~\ref{lem:biased-signs} gives the corresponding probability kernel and strongly continuous semigroup of positive contractions. Throughout this section, $C_u$, $R_t$, and $S_n=R_{2^{-n}}$, $n\in\N_0$, refer to this choice. Part~(\ref{item:biased-multiplier}) of that lemma and the resolvent formula~\eqref{eq:nu} give
\begin{equation*}
C_u\chi_F=e^{-u|F|}\chi_F,\qquad R_t\chi_F=(1+t|F|)^{-1}\chi_F.
\end{equation*}
Denote the corresponding space $\mathcal{X}_\lambda$ from~\eqref{eq:intrinsic} by $\mathcal{RB}$, and retain the maps $J$ and $K$ from~\eqref{eq:embedding}--\eqref{eq:retraction}. By Proposition~\ref{prop:projection}, $Q=JK$ is a projection on $\E$, with range $J\mathcal{RB}$ and $\norm{Q}\le4+3\sqrt2$. Its blocks and column estimates are given by~\eqref{eq:blocks} and~\eqref{eq:column}, with the present resolvents.

An operator between real $L_1$-spaces is \emph{non-Enflo} if it is bounded below on no subspace isomorphic to $L_1[0,1]$. We use the Enflo--Starbird characterization of these operators \cite{ES1979}, in the form stated in \cite[Theorem~1]{Starbird1976}, and Liu's theorem that they form a band in $\mathcal L(L_1[0,1])$ \cite[Lemma~2.14]{Liu1998}. Lemma~\ref{lem:block-band} will pass this property from individual blocks to the whole projection.

We shall apply the following lemma to the truncated resolvents.

\begin{lemma}\label{lem:smoothing-non-enflo}
Let $T\colon L_1(G,m)\to L_1(G,m)$ be an operator over the real field. If $\norm{Tf}_2\le c\norm{f}_p$ for every $f\in L_p(G,m)$, where $1<p<2$ and $c>0$, then $T$ is non-Enflo.
\end{lemma}
\begin{proof}
A \emph{bush} on $G$ is a refining sequence $(E_{n,k})_{k=1}^{m_n}$, $n\in\N_0$, of finite measurable partitions of a set $E_0$ of positive measure, with $m_0=1$ and $\delta_n=\max_{1\le k\le m_n}m(E_{n,k})\to0$ as $n\to\infty$. For such a bush, the assumed estimate gives
\begin{equation*}
\int_G\max_{1\le k\le m_n}|T\ind_{E_{n,k}}|\,dm \le\Big(\sum_{k=1}^{m_n}\norm{T\ind_{E_{n,k}}}_2^2\Big)^{1/2} \le cm(E_0)^{1/2}\delta_n^{1/p-1/2}\longrightarrow0 \quad\text{as }n\to\infty.
\end{equation*}
The Enflo--Starbird criterion therefore makes $T$ non-Enflo; see also~\cite[Definition~1.7 and Theorem~1.9]{Liu1998}.
\end{proof}

\begin{proof}[Proof of Theorem~\ref{thm:non-schur}]
We first work over the reals. For the present choice $\lambda_j=1$ for every $j\in\N$, the measure $\mu_u$ is Bonami's Riesz product with parameter $r=e^{-u}$. Fix $\varepsilon>0$ and put $p_\varepsilon=1+e^{-2\varepsilon}\in(1,2)$. With $q=2$, the condition $(q-1)r^2\le p_\varepsilon-1$ becomes $e^{-2u}\le e^{-2\varepsilon}$, which holds whenever $u\ge\varepsilon$. Bonami's hypercontractivity theorem and its proof~\cite[Chapter~III, Theorem~3]{Bonami1970} therefore give
\begin{equation*}
\norm{C_uf}_2\le\norm{f}_{p_\varepsilon}\qquad(f\in L_{p_\varepsilon}(G,m),\ u\ge\varepsilon>0).
\end{equation*}
For $s>0$, define the \emph{truncated resolvent} $R_{s,\varepsilon}\colon L_1(G,m)\to L_1(G,m)$ by
\begin{equation*}
R_{s,\varepsilon}f:=\int_\varepsilon^\infty s^{-1}e^{-u/s}C_uf\,du=e^{-\varepsilon/s}C_\varepsilon R_sf,
\end{equation*}
where last equality follows from the semigroup identity. For each $f\in L_1(G,m)$, the integral is understood in the Bochner sense. 

Contractivity of $C_u$ on $L_1$ and of $R_s$ on $L_{p_\varepsilon}$ therefore yields
\begin{equation*}
\norm{R_s-R_{s,\varepsilon}}_{1\to1} \le1-e^{-\varepsilon/s},\qquad \norm{R_{s,\varepsilon}f}_2 \le e^{-\varepsilon/s}\norm{f}_{p_\varepsilon},
\end{equation*}
so that, by Lemma~\ref{lem:smoothing-non-enflo}, $R_{s,\varepsilon}$ is non-Enflo for every $s,\varepsilon>0$.

In~\eqref{eq:blocks}, write $B_j=R_{\sigma_j}$, where $\sigma_0=1$ and $\sigma_j=2^{-(j-1)}$ for $j\ge1$. Each $A_i$ is bounded on both $L_1$ and $L_2$. For fixed $i,j\in\N_0$, replacing $B_j=R_{\sigma_j}$ by its truncation $R_{\sigma_j,\varepsilon}$ therefore gives a non-Enflo approximation to $Q_{ij}=2^{(i-j)/2}A_iB_j$, with $L_1$-operator error at most $2^{(i-j)/2}\norm{A_i}_{1\to1}(1-e^{-\varepsilon/\sigma_j}) \longrightarrow0$ as $\varepsilon\to0$. The non-Enflo class is norm closed, since a lower bound on a fixed subspace persists under a smaller operator-norm perturbation. Thus every $Q_{ij}$ is non-Enflo. The property survives bounded pre- and postcomposition: a composition bounded below on an $L_1$ copy forces its middle operator to be bounded below on the image of that copy under the first map. Consequently, the extended blocks on $\E$ are non-Enflo, and Lemma~\ref{lem:block-band} makes $Q$ non-Enflo. Since $Q$ is the identity on its range, $J\mathcal{RB}$ contains no copy of real $L_1[0,1]$.

To exhibit a copy of $\ell_2$ in $\mathcal{RB}$, write $r_j(x)=x_j$ for $j\in\N$. Over either scalar field, for every finite $F\subset\N$ and $(c_j)_{j\in F}\in\K^F$, the sum $f=\sum_{j\in F}c_jr_j$ satisfies
\begin{equation*}
\norm{f}_{\mathcal{RB}}=\gamma\norm{f}_1,\qquad \gamma=1+\sum_{n=1}^\infty \frac{2^{-n/2}}{1+2^{-n}}<\infty.
\end{equation*}
By Khintchine's inequality, the closed span of $(r_j)_{j\in\N}$ in $\mathcal{RB}$ is isomorphic to $\ell_2$. The corresponding embedding sends the unit-vector sequence of $\ell_2$ to the weakly null sequence $(r_j)_{j\in\N}$, whose norms equal $\gamma$, so $\mathcal{RB}$ is non-Schur.

By~\cite[Corollary~2.3]{DeHeviaLattices}, a complemented subspace of a real $L_1$-space that is isomorphic to a Banach lattice is isomorphic to an $L_1$-space. A separable infinite-dimensional $L_1$-space either is isomorphic to $\ell_1$ or contains $L_1[0,1]$. The copy of $\ell_2$ and the non-Enflo property of $Q$ exclude these alternatives, so the real range is not isomorphic to a Banach lattice.

For complex scalars, let $Q_{\mathbb R}$ be the real operator given by the same kernels. The realification of $J\mathcal{RB}$ is real-isomorphic, by $x+iy\mapsto(x,y)$, to $\operatorname{ran}Q_{\mathbb R}\oplus_1 \operatorname{ran}Q_{\mathbb R}$. After reindexing its two copies of $\E$, the projection $Q_{\mathbb R}\oplus Q_{\mathbb R}$ has only zero blocks and the non-Enflo blocks already considered. Lemma~\ref{lem:block-band} therefore excludes real $L_1[0,1]$ from its range. In particular, the complex range contains no complex $L_1[0,1]$. This realification is complemented in a real $L_1$-space and contains real $\ell_2$, so the same structural corollary also excludes a real Banach lattice isomorphism. Finally, the realification of a complex Banach lattice is real-isomorphic to the product of two copies of its real part, a real Banach lattice. Thus the complex range cannot be isomorphic to a complex Banach lattice either.
\end{proof}

\section{An interpolation interpretation}\label{subsec:interpolation}

We give an alternative interpretation of the constructions of Sections~\ref{sec:schur} and~\ref{sec:non-schur} through real interpolation. The argument applies to any sequence $\lambda=(\lambda_j)_{j\in\N}$ of nonnegative real numbers, as in Section~\ref{sec:general-construction}. Let $-A_\lambda$ be the generator of the associated semigroup $(C_u)_{u\ge0}$ on $L_1(G,m)$, so that $R_t=(I+tA_\lambda)^{-1}$, and give the domain of $A_\lambda$, denoted by $D(A_\lambda)$, its graph norm $\norm{g}_{D(A_\lambda)}=\norm{g}_1+\norm{A_\lambda g}_1$.

\begin{proposition}\label{prop:interpolation}
As spaces of functions, with equivalent norms,
\begin{equation*}
\mathcal{X}_\lambda=\bigl(L_1(G,m),D(A_\lambda)\bigr)_{1/2,1}.
\end{equation*}
\end{proposition}
\begin{proof}
Write
\begin{equation*}
\mathcal K(t,f)=\inf_{g\in D(A_\lambda)} \bigl(\norm{f-g}_1+t\norm{g}_{D(A_\lambda)}\bigr).
\end{equation*}
For $0<t\le1$, the resolvent identities give
\begin{equation}\label{eq:k-functional}
\tfrac12\norm{(I-R_t)f}_1 \le\mathcal K(t,f) \le2\norm{(I-R_t)f}_1+t\norm{f}_1.
\end{equation}
For the upper bound, take $g=R_tf$ and use $tA_\lambda R_t=I-R_t$ and contractivity. For the lower bound, apply $I-R_t$ to $f=(f-g)+g$ and use $\norm{(I-R_t)g}_1\le t\norm{A_\lambda g}_1$. For $t\ge1$, the graph norm dominates the $L_1(G,m)$ norm, so $\mathcal K(t,f)=\norm{f}_1$. Since $\mathcal K(t,f)$ is nondecreasing and $\mathcal K(t,f)/t$ is nonincreasing, the interpolation integral is comparable with its dyadic samples:
\begin{equation*}
\int_0^\infty t^{-1/2}\mathcal K(t,f)\,\frac{dt}{t} \asymp\norm{f}_1+ \sum_{n=1}^\infty2^{n/2}\mathcal K(2^{-n},f).
\end{equation*}
Here $\asymp$ means that each side is bounded by a positive constant multiple of the other, with constants independent of $f$. Equation~\eqref{eq:k-functional} and summability of $2^{-n/2}$ identify this expression, up to fixed constants, with $\norm{f}_{\mathcal{X}_\lambda}$.
\end{proof}

In particular, this identifies $\mathcal{CS}$ and $\mathcal{RB}$ with the interpolation spaces corresponding to $\lambda_j=\sqrt{\log(j+1)}$ and $\lambda_j=1$ for $j\in\N$, respectively.

For either choice of $\lambda$, the retraction also applies on $L_p(G,m)$, $1\le p<\infty$. Let $\mathcal{X}_{p,\lambda}$ consist of the functions $f\in L_p(G,m)$ for which
\begin{equation*}
\norm{f}_{\mathcal{X}_{p,\lambda}} =\left(\norm{f}_p^p+ \sum_{n=1}^{\infty}2^{np/2}\norm{(I-S_n)f}_p^p\right)^{1/p} <\infty.
\end{equation*}
Thus $\mathcal{X}_{1,\lambda}=\mathcal{X}_\lambda$. The maps in~\eqref{eq:embedding}--\eqref{eq:retraction} complement $\mathcal{X}_{p,\lambda}$ in $\mathcal E_p=\ell_p\bigl(\N_0;L_p(G,m)\bigr)\cong L_p[0,1]$: the geometric block estimates give uniformly bounded row and column sums. The preceding $\mathcal K$-functional argument identifies $\mathcal{X}_{p,\lambda}=\bigl(L_p(G,m),D(A_{p,\lambda})\bigr)_{1/2,p}$ with equivalent norms, where $-A_{p,\lambda}$ is the generator on $L_p(G,m)$ and its domain is equipped with the graph norm.

For $1<p<\infty$, it can be checked that the choice of $\lambda$ from Section~\ref{sec:schur} gives $\mathcal{X}_{p,\lambda}\cong\ell_p$. For $\lambda_j=1$ for every $j\in\N$, the space $\mathcal{X}_{p,\lambda}$ contains a copy of $\ell_2$ and is not isomorphic to $\ell_p$ when $p\ne2$. For both choices, all the ranges with $1<p<\infty$ are reflexive and have the Radon--Nikod\'ym property.

\par\medskip\noindent\textbf{AI statement.}
The author made substantial use of OpenAI's ChatGPT 5.6 Sol and 6.0 Astra in developing this work. In particular, the main constructions arose in conversation with these models.

\par\medskip\noindent\textbf{Acknowledgements.}
The author acknowledges with thanks funding from the EPSRC (grant number EP/W524438/1) that has supported his PhD studies.

\end{document}